\documentclass[12pt,a4paper]{amsart}
\usepackage[latin1]{inputenc}
\usepackage{mathptmx}
\usepackage{amssymb}
\usepackage{hyperref}
\usepackage{mathrsfs}
\usepackage[mathscr]{euscript}
\usepackage{booktabs}

\usepackage{todonotes}

\usepackage[ignoreunlbld,norefs,nocites,nomsgs]{refcheck}

\usepackage{pgf, tikz}

\usepackage{fancyvrb} \RecustomVerbatimEnvironment{Verbatim}{Verbatim}
{xleftmargin=15pt, frame=single, fontsize=\small}

\usepackage{color}

\definecolor{light}{gray}{0.9}
\definecolor{medium}{gray}{0.8}

\newtheorem{theorem}{Theorem}

\newtheorem{proposition}[theorem]{Proposition}

\theoremstyle{definition}
\newcounter{rem}
\newcounter{def}
\newtheorem{remark}[rem]{Remark}
\newtheorem{definition}[def]{Definition}

\numberwithin{equation}{section}

\def\ZZ{{\mathbb Z}}
\def\NN{{\mathbb N}}
\def\RR{{\mathbb R}}
\def\QQ{{\mathbb Q}}
\def\11{{\mathbb 1}}

\def\cP{{\mathscr P}}

\def\cQ{{\mathscr Q}}
\def\cB{{\mathscr B}}

\def\cE{{\mathscr E}}
\def\cF{{\mathscr F}}

\def\Sg{{\textup {Sg}}}
\def\SgRev{{\textup {SgRev}}}
\def\RevNPR{{\textup {RevNPR}}}
\def\RevPR{{\textup {RevPR}}}
\def\Rev{{\textup {Rev}}}
\def\PR{{\textup {PR}}}
\def\NPR{{\textup {NPR}}}

\def\CW{\textup{CW}}

\def\GL{\operatorname{GL}}

\def\conv{\operatorname{conv}}

\def\vol{\operatorname{vol}}

\def\codim{\operatorname{codim}}

\def\ttt#1{\texttt{#1}}

\let\epsilon=\varepsilon

\allowdisplaybreaks

\def\Vol{\operatorname{Vol}}
\def\aff{\operatorname{aff}}
\def\Ht{\operatorname{Ht}}

\def\cD{{\mathcal D}}
\def\strut{\vphantom{\Large(} }

\begin{document}

\title[Polytope volume by descent and applications in social choice]{Polytope volume by descent in the face lattice\\ and applications in social choice}
\author[W. Bruns]{Winfried Bruns}
\address{Winfried Bruns\\ Universit\"at Osnabr\"uck\\ FB Mathematik/Informatik\\ 49069 Osna\-br\"uck\\ Germany}
\email{wbruns@uos.de}
\author[B. Ichim]{Bogdan Ichim}
\address{Bogdan Ichim \\ University of Bucharest \\ Faculty of Mathematics and Computer Science \\ Str. Academiei 14 \\ 010014 Bucharest \\ Romania \\ and \newline
 Simion Stoilow Institute of Mathematics of the Romanian Academy \\ Research Unit 5 \\ C.P. 1-764 \\ 010702 Bucharest \\ Romania}
\email{bogdan.ichim@fmi.unibuc.ro \\ bogdan.ichim@imar.ro}



\thanks{The second author was partially supported by a grant of Romanian Ministry of Research and Innovation, CNCS - UEFISCDI, project number PN-III-P4-ID-PCE-2016-0157, within PNCDI III}

\maketitle

\begin{abstract}
We describe the computation of polytope volumes by descent in the face lattice, its implementation in Normaliz, and the connection to reverse-lexicographic triangulations. The efficiency of the algorithm is demonstrated by several high dimensional polytopes of different characteristics. Finally, we present an application to voting theory  where polytope volumes appear as probabilities of certain paradoxa.
\end{abstract}


\section{Introduction}

The volume of a polytope is a geometric magnitude that has been studied since antiquity~-- formulas for the area of polygons or the volume of common $3$-dimensional polytopes like pyramids were known in Babylonian and Egyptian mathematics. From a modern viewpoint one can say that these formulas are recursive:  they reduce volume computations to the measurement of lengths.

In toric algebra and geometry, volumes are almost a synonym for multiplicities, defined as leading coefficients of Hilbert (quasi)polynomials. For this classical connection see Teissier \cite{Teiss} or Bruns and Gubeladze \cite[Section 6.E]{BG}.
The package Normaliz \cite{Nmz} has contained options for the computation of Hilbert series and multiplicities from its beginnings. Until recently, the only approach to volumes was based on lexicographic (or placing) triangulations: the volume of a polytope is obtained as sum of simplex volumes, and the volume of a simplex is essentially computed as a determinant. (We refer the reader to \cite{BG} for unexplained algebraic notions and to Section \ref{polytopes} for geometric terminology.) The algorithm has been improved in several steps; see \cite{BIS} and \cite{BSS} for a description of the present state and \cite{BS} for a refined version computing integrals.

An attractive application of polytope volumes is probabilities of certain events in voting theory where results of elections with $n$ candidates can be identified with lattice points in the positive orthant of $\RR^N$, $N=n!$. Several classical phenomena, most prominently the Condorcet paradox, can be described by homogeneous linear inequalities. In a suitable probabilistic model, their probabilities for a large number of voters can be computed as lattice normalized volumes of rational polytopes. Already for $n=4$ one has $N=24$, so that these computations pose a challenging problem. Nevertheless, quite a number of interesting volumes and Hilbert series have been determined via lexicographic triangulations \cite{BIS2}.

When more complicated computations of voting theory turned out inaccessible for the approach by lexicographic triangulations, we decided to add an algorithm for the volume of a polytope $P$ that is based on descent in the face lattice of $P$. (The twofold meaning of ``lattice'' is an unfortunate, if customary clash of terminology.) In principle, it follows the classical formulas in reducing the computation of volumes to the measurement of lengths. However, we take a hybrid approach, and compute the volume of a simplex as a determinant. For this reason our descent algorithm behaves quite well also for simplicial polytopes.

We have not yet made precise what we mean by ``volume''. From our viewpoint, the central invariant for rational polytopes is \emph{lattice (normalized) volume}, for which a unimodular lattice simplex $\Sigma$ has volume $1$ in the affine space spanned by $\Sigma$. Lattice normalized volume in $\RR^n$ is invariant under the action of $\GL(n,\ZZ)$. This group can be exactly represented in rational arithmetic. In contrast to Euclidean volume, which is invariant under the orthogonal group $O(n)$, it is a rational number also for lower dimensional polytopes (i.e., polytopes of dimension smaller than that of the ambient space). Therefore it is well-suited for precise computation by recursion to faces of a polytope. At the end, the Euclidean volume can be (and is) obtained by conversion from the lattice volume.

When the height $\Ht_F(v)$ of a point $v$ over a facet is also measured by the lattice, then the recursion formula takes the simple form
$$
\Vol(P)=\sum_{F\text{ facet of }P} \Ht_F(v)\Vol(F)
$$
for a point $v\in P$. We call a subset $\cD$ of the face lattice a \emph{descent system} if it contains all faces that come up in the successive application of this formula. We try to keep $\cD$ small by carefully choosing the points $v$, but the polytope should not have too many non-simplex faces to keep the descent system from explosion.  The largest descent system that has been computed for this paper has cardinality $> 6\cdot 10^8$. This number makes it clear that a careful balance between memory usage and computation time is imperative if a wide range of applications is desired.

Though it is not visible straightaway, there is a triangulation or at least a similar  decomposition in the background of the descent algorithm, namely a reverse lexicographic (or pulling) triangulation. The crucial point is \emph{not} to compute it explicitly, but to distill the volume relevant information into a weight function associated with the descent system.

Normaliz is not the first package to exploit descent in the face lattice for volume computations. In B\"{u}eler and Enge's package Vinci \cite{vinci} it appears as the \emph{hot} algorithm, where ``hot'' is an acronym for ``hybrid orthonormalization  technique''. Vinci is based on the article \cite{practical} by Büeler, Enge and Fukuda. In contrast to Normaliz, Vinci uses floating point arithmetic and computes only Euclidean volumes, and it does not contain convex hull computation or vertex enumeration, so that another package is needed for this auxiliary task.

In their very recent paper \cite{Emiris} Emiris and Fisikopoulos discuss probabilistic methods for estimating the volume of a polytope. On p.~ 38:2 they say that several packages, including Normaliz, ``cannot handle general polytopes for dimension $d > 15$''. Based on our experience, we cannot fully support this finding.

Section \ref{polytopes} introduces the geometric terminology used in the sequel. In Section \ref{LattVol} we discuss lattice volume, height and the linear algebra aspects of the algorithm. Section \ref{system} describes the algorithm and its connection with reverse lexicographic triangulations. We give sample computations in Section \ref{samples}. They are based on Normaliz 3.6.1\footnote{released July 7, 2018} and include a wide range of polytopes with quite different characteristics and of dimensions between $15$ and $55$. The descent algorithm has staid essentially unchanged since version 3.6.1.

Section \ref{samples} contains also a comparison of the two Normaliz algorithms (triangulation and descent) to vinci's ``hot'' (descent) and ``rch''  (reverse lexicographic triangulation), and additionally to vinci's ``rlass'' (a revised version of Lassere's algorithm \cite{lass}) that is based on the same type of recursion, but has no Normaliz analogue.  Despite of the same basic algorithmic approach of Normaliz' descent and vinci's hot, the implementations in Normaliz and vinci differ substantially, with significant consequences for applicability. A zip file with input for all computations of this paper can be downloaded from
\begin{center}
\url{https://www.normaliz.uni-osnabrueck.de/wp-content/uploads/2020/03/PolytopeVolumes.zip}.
\end{center}

Finally, Section \ref{Volumes Computations} discusses applications to voting theory. We have included it to demonstrate the efficiency of our algorithm in the field of voting theory. It provided many challenges for us without which this algorithm would not have been developed. We quote from  \cite{Ouafdi}: \emph{It seems, however, that the latest version of Normaliz is, at the present time, the most efficient software tool to obtain the IAC probabilities of electoral outcomes	when more than three alternatives are in contention \dots}
		
\section{Polytopes and cones}\label{polytopes}

A \emph{polytope} $P\subset \RR^n$ is the convex hull of finitely many points $v_1,\dots,v_m\in\RR^n$:
$$
P=\conv(v_1,\dots,v_m)=\Bigl\{ \alpha_1v_1+\dots+ \alpha_mv_m: 0\le \alpha_i\le 1, i=1\dots, m, \sum_{i=1}^{m} \alpha_i=1   \Bigr\}.
$$
By a fundamental theorem of Minkowski, a polytope can equivalently be described as a compact set that is the intersection of finitely many affine halfspaces:
$$
P=\bigcap_{j=1}^s H_i^+,
$$
where an \emph{affine halfspace} is a set
$$
H^+=\{ x\in\RR^n: \lambda(x)\ge \beta \}
$$
for a nonzero linear form $\lambda\in(\RR^n)^*$ and $\beta \in\RR$. (Not necessarily compact intersections of halfspaces are called \emph{polyhedra}.) The halfspace $H^+$ is bounded by the hyperplane $H=\{x\in\RR^n:  \lambda(x)= \beta \}$.

 The two descriptions of polytopes are often called \emph{V-representations} and \emph{H-represent\-ations}. If $v_1,\dots,v_m\in\QQ^n$, then $P$ is a \emph{rational polytope}; equivalently one can choose the halfspaces to be \emph{rational}: $\lambda$ has coefficients in $\QQ$ and $\beta\in\QQ$. The \emph{dimension} of $P$ is the dimension of its affine hull $\aff(P)$. The polytope $P\subset \RR^n$ is \emph{full dimensional} if $\dim P=n$.

A \emph{cone} $C$ is the conical hull of finitely many vectors:
$$
C=\bigl\{ \alpha_1v_1+\dots+ \alpha_mv_m:\alpha_i\ge 0, i=1\dots, m \bigr\}.
$$
Equivalently, it is the intersection of finitely many \emph{linear halfspaces} in whose definition $\beta=0$. The attribute \emph{rational} is given in analogy with polyhedra. A cone is \emph{pointed} if $\{0\}$ is the only vector subspace contained in $C$. For computations one usually represents a polytope $P\subset\RR^n$ as the intersection of a pointed cone $C\subset\RR^{n+1}$ with an affine hyperplane $H\cong \RR^n$. We introduce this technique in Section \ref{LattVol}; see the discussion surrounding Equation \eqref{homog}.

Let $P$ be a polytope or cone (or a polyhedron in general). If $P$ is contained in  one of the two halfspaces bounded by a hyperplane $S$, then $S$ is called a \emph{support hyperperplane} of $P$. A subset $F\subset P$ is called a \emph{face} if $F=P\cap S$ for a support hyperplane $S$ of $P$. The polytope $P$ is an improper face of $P$, as well as $\emptyset$ if $P$ is a polytope. There are only finitely many faces, and every face is a polytope or cone, respectively. The maximal proper faces of $P$ are called \emph{facets}. The faces of $P$ are partially ordered by inclusion. They actually form a lattice, the \emph{face lattice} of $P$. Every face $F\neq P$ is the intersection of facets. The properties of the face lattice are important in the following; for the details see Ziegler \cite{Zi}. The only point in a face of dimension $0$ is a \emph{vertex} of $P$. A face of dimension $1$ is an \emph{edge} if $P$ is a polytope and an \emph{extreme ray} if $P$ is a pointed cone. The difference $\codim F=\dim P-\dim F$ is called the \emph{codimension} of $F$. We say that $F$ is a \emph{subfacet} if it has codimension $2$. Often subfacets are called \emph{ridges}.

The vertices of a polytope $P$ form the unique minimal set $V$ of points such that $P=\conv(V)$. Similarly a set $E$ containing exactly one vector from each extreme ray is  a unique minimal generating set of a pointed cone. It is clear that one wants to work with such minimal descriptions in computations.

In the following we use the term ``support hyperplane'' in a restricted sense: we additionally require that a support hyperplane $S$ intersects $P$ in a facet. This is justified: $P=\aff(P)\cap H_1^+\cap\dots\cap H_s^+$ where $H_1^+,\dots,H_s^+$ are halfspaces whose support hyperplanes intersect $P$ in pairwise different facets. The intersections $\aff(P)\cap H_i^+$ in such an irredundant representation are uniquely determined. In the fulldimensional case the halfspaces themselves are therefore uniquely determined.

A \emph{simplex} is a polytope of dimension $d$ with exactly $d+1$ vertices: a triangle in dimension $2$, a tetrahedron in  dimension $3$ etc. A polytope is \emph{simplicial} if all its facets are simplices. If every vertex of a polytope $P$ of dimension $d$ is contained in exactly $d$ facets, then $P$ is \emph{simple}. It follows that every face is contained in exactly $\codim F$ facets.

\section{Lattice normalized volume}\label{LattVol}

As mentioned in the introduction, we compute the lattice normalized volume of a rational polytope $P\subset\RR^n$, i.e., a polytope with vertices in $\QQ^n$. Let us explain this notion.
The affine hull $A=\aff(P)$ is a rational affine subspace of $\RR^n$. First assume that $0\in A$. Then $L=\aff(P)\cap\ZZ^n$ is a subgroup of $\ZZ^n$ of rank $d=\dim P$ (and $\ZZ^n/L$ is torsionfree). Choose a $\ZZ$-basis $v_1,\dots,v_d$ of $L$. The \emph{lattice (normalized) volume} $\Vol$ on $A$ is the Lebesgue measure on $A$ scaled in such a way that the simplex $\conv(0,v_1,\dots,v_d)$ has measure $1$. The definition is independent of the choice of $v_1,\dots,v_d$ since all invertible $d\times d$ matrices over $\ZZ$ have determinant $\pm 1$. If $0\notin A$, then we replace $A$ by a translate $A_0=A-w$, $w\in A$, and set $\Vol(X)=\Vol(X-w)$ for  $X\subset A$. This definition is independent of the choice of $w$ since $\Vol$ is translation invariant on $A_0$. Note that the polytope containing a single point $x\in\QQ^n$ has lattice volume $1$. (If desired, the definition of lattice volume can be extended to arbitrary measurable subsets of $A$.)

If $P$ is a lattice polytope, i.e., a polytope with vertices in $\ZZ^n$, then $\Vol(P)$ is an integer. For an arbitrary rational polytope we have $\Vol(P)\in\QQ$. As a consequence, $\Vol(P)$ can be computed precisely by rational arithmetic.

If $P$ has full dimension $n$, then $\Vol(P)=n!\vol(P)$ where $\vol$ denotes the Euclidean volume. So it is only a matter of scaling by the integer $n!$ whether one computes the lattice volume or the Euclidean volume. However, if $\dim P<n$, then $\vol(P)$ need not be rational anymore: the diagonal in the unit square has Euclidean length $\sqrt 2$, but lattice length  $1$.

A second invariant we need is the lattice height of a \emph{rational} point $x$ over a rational subspace $H\neq\emptyset$. More generally, one can consider points $x$ such that $\aff(H,x)$ is again rational; for example, this is the case if $H$ is a hyperplane in $\RR^n$. If $x\in H$, we set $\Ht_H(x)=0$. Otherwise let $A=\aff(x,H)$ so that $H$ is a hyperplane in $A$. Assume first that $0\in A$.

Then $H$ is cut out from $A$ by an equation $\lambda(y)=\beta$ with a primitive $\ZZ$-linear form $\lambda$ on $L=A\cap \ZZ^n$ and $\beta\in\QQ$. That $\lambda$ is \emph{primitive} means that there exists $y\in L$ such that $\lambda(y)=1$. Note that $\lambda$ can be chosen such that it has integral values not only on $L$, but on the whole of $\ZZ^n$. In fact, $L$ is not an arbitrary sublattice, but $\ZZ^n/L$ is torsionfree: $L$ is the intersection of a vector subspace of $\RR^n$ with $\ZZ^n$. This implies that $L$ has a complement $M$ in $\ZZ^n$ such that $\ZZ^n=L\oplus M$. So $\lambda$ can be extended as a $\ZZ$-linear form on $\ZZ^n$ by the linear form $0$ on $M$.

With this choice of $\lambda$, $\Ht_H(x)=|\lambda(x)-\beta|$ is called the\emph{ lattice height} of $x$ over $H$. (There are exactly two choices for the pair $(\lambda,\beta)$, differing by the factor $-1$.) If $0\notin A$, then we choose an auxiliary point $v\in A$, replace $H$ by $H-v$, $A$ by $A-v$ and $x$ by $x-v$. (In the algorithm we will only have to deal with the case $0\in H$.)
If $P$ is a rational polytope  and $F$ is a facet or, more generally, a face of $P$, then $\Ht_F(x)=\Ht_H(x)$ where $H=\aff(F)$.

In order to compute the lattice height $\Ht_{\{0\}}(x)$ of a rational point $x$ over the origin, one considers the ray from $0$ through $x$ and chooses $u$ as the first nonzero integer point on this ray. Then $v=au$ for some $a\in\QQ$, and $\Ht_{\{0\}}(x)=\Ht_{\{x\}}(0) =a$.

In Figure \ref{Ht_fig} we have chosen $v=(1/2,1)$ and $w=(-1/2,1)$. The primitive linear forms defining $E$ and $F$ are $\lambda(x,y)=y$ and $\mu(x,y)=-2x+y$, respectively. Note that for measuring $\Ht_{\{v\}}(w)$ we must replace $\mu$ by $\mu/2$ since $\mu(s)$ and $\mu(t)$ differ by $\pm2$ for successive lattice points $s$ and $t$ on the line through $w$ and $v$.
Analogously, for measuring  $\Ht_{\{v\}}(0)$ we must replace $\lambda$ by $\lambda/2$ since $\lambda (s)$ and $\lambda(t)$ differ by $\pm2$ for successive lattice points $s$ and $t$ on the line through $0$ and $v$. We obtain
\begin{gather*}
\Ht_{F}(w)=2, \qquad \Ht_{\{v\}}(w)=1,\qquad \Ht_{\{v\}}(0)=1/2,\qquad \Ht_{E}(0)=1,\\
\Vol(P)=1,\qquad \Vol(E)=1,\qquad \Vol(F)=1/2.
\end{gather*}
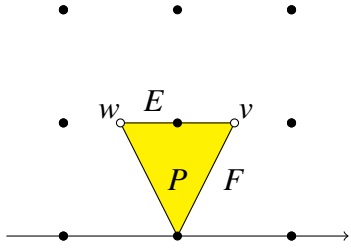
\begin{figure}[hbt]
\begin{center}
\begin{tikzpicture}[scale=1.5]

\filldraw[yellow] (0,0) -- (0.5,1) -- (-0.5,1) -- cycle;
\draw (0,0) -- (0.5,1) -- (-0.5,1) -- cycle;
\draw[->] (-1.5,0) -- (1.5,0);

\foreach \x in {-1,-1,0,1,1}
\foreach \y in {0,1,2}
{
	\filldraw[fill=black] (\x,\y)  circle (1pt);
}
\filldraw[fill=white] (-0.5,1) circle (1pt);
\filldraw[fill=white] (0.5,1) circle (1pt);
\draw node at (0.6, 1.1){$v$};
\draw node at (-0.6, 1.1){$w$};
\draw node at (0.5, 0.5){$F$};
\draw node at (-0.2,1.2){$E$};
\draw node at (0,0.5){$P$};
\end{tikzpicture}
\end{center}
\caption{A rational polytope}\label{Ht_fig}
\end{figure}

\begin{proposition}\label{decomp}
Let $P\subset\RR^n$ be a rational polytope, and $v\in P$. Then
\begin{equation}
\Vol(P)=\sum_{F\text{ facet of }P} \Ht_F(v)\Vol(F).\label{Vol_dec}
\end{equation}
\end{proposition}

\begin{proof}
$P$ is the union of the ``pyramids'' $\conv(v,F)$ where $F$ runs through the facets of $P$. These pyramids intersect in lower dimensional polytopes $\conv(v,G)$ where $G$ is a face of codimension $\ge 2$ of $P$. Since the intersections have measure $0$ in the Lebesgue measure on $\aff(P)$ independently of the scaling, the proposition follows by the additivity of the measure, provided $\Vol(\conv(v,F))=\Ht_F(v)\Vol(F)$ for all facets $F$ of $P$.

To prove this claim, we can triangulate $F$, and use additivity again, thereby reducing it to the case of a pyramid over a simplex $\Delta$, $\dim\Delta=\dim P-1$. Then we choose a positive integer $k$ such that $kv$ and all vertices of $k\Delta$  have integer coordinates. This scales $\Vol(\conv(v,\Delta))$ as well as $\Ht_\Delta(v)\Vol(\Delta)$ by the factor $k^d$, $d=\dim P$. Therefore we can finally assume that $v$ is a vertex of a lattice simplex $P$ and $F=\Delta$ is its opposite facet. Under these hypotheses \cite[3.9]{BG} says exactly what we need: $\Vol(P)=\Ht_F(v)\Vol(F)$.  In \cite{BG} the lattice volume of a lattice polytope is called its \emph{multiplicity}; we will explain this terminology in Remark \ref{rem_algebra}. For the convenience of the reader we include a proof of \cite[3.9]{BG} below.
\end{proof}

Figure \ref{dec_fig} illustrates Proposition \ref{decomp} in a simple case. It is clear that one should take $v$ as a vertex of $P$ in order to minimize the number of nonzero summands in Equation \eqref{Vol_dec}. Our choice of $v$ will be discussed in the next section.
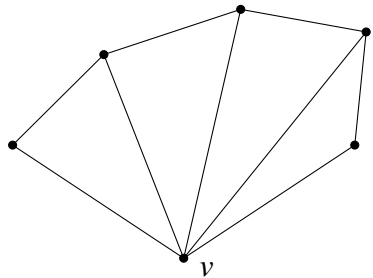
\begin{figure}[hbt]
\begin{center}
\begin{tikzpicture}[scale=1.5]
\draw (0,0) -- (1.5,1) -- (1.6,2) -- (0.5,2.2) -- (-0.7,1.8) -- (-1.5,1) -- cycle;
\draw (0,0) -- (0.5,2.2);
\draw (0,0) -- (1.6,2);
\draw (0,0) -- (-0.7,1.8);
\draw node at (0.2,-0.1){$v$};
\filldraw[fill=black] (0,0)  circle (1pt);
\filldraw[fill=black] (1.5,1)  circle (1pt);
\filldraw[fill=black] (1.6,2)  circle (1pt);
\filldraw[fill=black] (0.5,2.2)  circle (1pt);
\filldraw[fill=black] (-0.7,1.8)  circle (1pt);
\filldraw[fill=black] (-1.5,1)  circle (1pt);
\end{tikzpicture}
\end{center}
\caption{Decomposition of a polygon into pyramids}\label{dec_fig}
\end{figure}

\begin{remark}
The proposition holds for all $v\in\aff(P)$, provided we replace $\Ht$ by its signed variant: In the definition choose the sign of $\lambda$ in such a way that $\lambda(x)-\beta\ge 0$ for $x\in P$ and set $\Ht_F(y)=\lambda(y)-\beta$ for $y\in\aff(P)$.  This is important if one wants to represent $P$ by the signed decomposition into the pyramids $\conv(v,F)$. Normaliz does not use signed decompositions at present.
\end{remark}

Together with the observation that a single point has lattice volume $1$, the proposition constitutes a complete, recursive  algorithm for the computation of lattice volumes, provided one can compute lattice height. In principle, this is the algorithm that we have implemented. However, an implementation of any practical value requires considerable care, as we will see in the next section.

The first useful modification is to stop the recursion if one hits a simplex face in the descent, and to compute the volume of a simplex directly.

\begin{proposition}\label{simplex}
Let $\Delta\subset\RR^n$ be a rational simplex with vertices $v_0,\dots,v_d$. Choose a basis $u_1,\dots,u_d$ of the lattice $\aff(\Delta-v_0)\cap \ZZ^n$. Define the $d\times d$ matrix $T=(t_{ij})$  by the representations $v_i-v_0=\sum_{j=1}^{d} t_{ij}u_j$, $i=1,\dots,d$. Then
$$
\Vol(\Delta)=|\det(T)|.
$$
\end{proposition}

\begin{proof}
This follows immediately from the substitution rule for Lebesgue integrals applied to the constant function $f=1$. See \cite[2.C]{BG} for an algebraic proof.
\end{proof}

\begin{remark}\label{Delta_proof}
(a) With the help of Proposition \ref{simplex} one can easily complete the proof of Proposition \ref{decomp} without a reference to \cite{BG}. After a parallel translation we can assume that $v_0=0$ is a vertex of $\Delta$ and $v=v_d$ ($v$ as as in Proposition \ref{simplex}). Let $U$ be the $\QQ$-vector subspace spanned by $v_1,\dots,v_d$ and $U'$ its subspace spanned by $v_1,\dots,v_{d-1}$. Then we define sublattices $L\subset \ZZ^n$ and $L'\subset L$ by $L=\ZZ^n\cap U$ and $L'=\ZZ^n\cap U'$. Evidently $L/L'$ is torsionfree. In other words, $L'$ is a direct summand of $L$. This means, we can complete a $\ZZ$-basis $u_1,\dots,u_{d-1}$ of $U'$ by $u_d\in L$ to a $\ZZ$-basis of $L$.

Let $\Delta'$ be the $(d-1)$-subsimplex $\conv(0,v_1,\dots,v_{d-1})$. With the notation of Proposition \ref{simplex}, we obtain $\Vol(\Delta)=|t_{dd}|\Vol(\Delta')$ by Laplace expansion of $\det(T)$ along the last column of $T$. For Proposition \ref{decomp} it remains to show that $|t_{dd}|=\Ht_{\Delta'}(v_d)$. In fact,
the primitive $\ZZ$-linear form $\lambda$  on $L$ that vanishes on $L'$ has value $\pm 1$ on $u_d$. So $|\lambda(v_d)|=|t_{dd}|$.

(b) As an anonymous referee pointed out, one can prove Proposition \ref{decomp} by using the notion of determinant of an affine lattice; see Martinet \cite[Prop. 1.3.4]{Martinet}. This approach reduces Proposition \ref{decomp} to an assertion about Euclidean volume.
\end{remark}

If we follow the definition of lattice height, then it is clear that we must choose a vertex of $F$ as the origin of the coordinate system for every face $F$ that comes up in the recursive application of Proposition \ref{decomp}. This complication disappears if $0\in\RR^n$ is a vertex of every face $F$ involved. The reduction of the general case to the special situation is by the standard operation of homogenization.

Normaliz represents a rational polytope $P$ in homogeneous coordinates as follows: $C$ is a pointed cone generated by integral vectors $v_1,\dots,v_m$, $\delta$ is a primitive $\ZZ$-linear form on $\ZZ^n$ such that $\delta(x)>0$ for all $x\in C$, $x\neq 0$,  and
\begin{equation}
P=\{x\in C: \delta(x)= 1\}.\label{homog}
\end{equation}
(In the terminology of \cite{BG}, $\delta$ defines a grading on $\ZZ^n$.) If $P$ is not already given in this form, we can easily realize it as such by introducing a homogenizing $(n+1)$-th coordinate: we replace $P\subset\RR^n$ by $P'=P\times \{1\}\subset \RR^{n+1}$, set $C=\RR_+ P'$ and $\delta(x)=x_{n+1}$ for $x=(x_1,\dots,x_{n+1})$. Consequently one can directly assume that $P$ is given by Equation \eqref{homog}. Then it is natural to pass to $\overline P=\conv(0,P)$. All faces of $\overline P$ have $0$ as a vertex, except $P$ and its faces. Under a mild condition we have $\Vol(P)=\Vol(\overline P)$, and in the general case we can easily find the correcting factor:

\begin{proposition}\label{nondegerate}
Suppose that the rational polytope $P$ is given as in Equation \eqref{homog}. If $\aff(P)$ contains an integral point, then $\Vol(P)=\Vol(\overline P)$. More generally, if $k$ is the smallest positive integer such that $\aff(kP)$ contains an integral point, then $\Vol(P)=k\Vol(\overline P)$.
\end{proposition}

\begin{proof}
It is enough to discuss the general case.
That the parallels $\aff(jP)$, $1\le j <k$,  do not contain lattice points implies that $k\mid \delta(x)$ for all $x\in L=\aff(\overline P)\cap \ZZ^n$. On the other hand, $\delta(y)=k$ for some $y\in\aff(kP)\cap L$. Therefore $\delta/k$ is a primitive linear form on $L$. Clearly $\Ht_{P}(0)=1/k$, and we can apply  Proposition \ref{decomp}.
\end{proof}

If, in the situation of Proposition \ref{nondegerate}, $\dim P=n-1$ or $P$ itself contains a lattice point, then evidently $\Vol(P)=\Vol(\overline P)$.

The discussion above is summarized in the next proposition. It describes exactly the arithmetic of the practical computation.

\begin{proposition}\label{arithmetic}
Let $v_1,\dots, v_m\in\ZZ^n$ and suppose $C=\RR_+v_1+\dots+\RR_+v_m$ is a pointed cone, generated by $v_1,\dots,v_m$. Let $\delta$ be a primitive $\ZZ$-linear form on $\ZZ^n$ with $\delta(x)>0$ for all nonzero  $x\in C$. Set $P=\{x\in C:\delta(x)=1\}$ and suppose that $d=\dim P\ge 1$. As above, we set $\overline P=\conv(0,P)$.
\begin{enumerate}
\item
For $v\in C$, $v\neq 0$, one has
$$
\Vol(\overline P)=\frac{1}{\delta(v)}\sum_{F\text{ facet of }P}\Ht_{\overline F}(v)\Vol(\overline F).
$$

\item Let $F$ be a facet of $P$ and $\overline F=\conv(0,F)$ the corresponding facet of $\overline P$. Suppose that $\overline F$ is cut out from $\overline P$ by the $\ZZ$-linear form $\lambda$ with $\lambda(x)\ge 0$ for $x\in\overline P$. Let $u_1,\dots,u_{d+1}$  be a $\ZZ$-basis of $\aff(\overline P)\cap \ZZ^n$ and set $g=\gcd(\lambda(u_1),\dots,\allowbreak \lambda(u_{d+1}))$.
Then $\Ht_{\overline F}(v)=\lambda(v)/g$.

\item Suppose that $m=d+1$ and that $v_1,\dots,v_{m}$ are linearly independent.  With $u_1,\dots,\allowbreak u_{d+1}$ as in \textup{(2)}, one has
$$
\Vol(\overline P)=\frac{1}{\delta(v_1)\cdots\delta(v_{m})}|\det(T)|,
$$
where $T=(t_{ij})$ is the $m\times m$-matrix defined by $v_i=\sum_{j=1}^m t_{ij}u_j$, $i=1,\dots,m$.
\end{enumerate}
\end{proposition}

\begin{proof}
For (1) we observe that $v/\delta(v)\in P$. Therefore $\Ht_P(v/\delta(v))=0$ and we need to sum in Proposition \ref{decomp} (applied to $v/\delta(v)$) only over the facets $\neq P$ of $\overline P$. They are exactly the facets of type $\overline F$ with $F$ a facet of $P$. Since $0\in \overline F$, the function $\Ht_{\overline F}$ is linear so that $\Ht_{\overline F}(v/\delta(v))=\Ht_{\overline F}(v)/\delta(v)$.

For (2) we note that the primitive linear form on $\aff(\overline P)\cap \ZZ^n$ that computes $\Ht_{\overline F}$ is indeed $\lambda/g$.

(3) follows from Proposition \ref{simplex} after scaling $v_1,\dots,v_m$ by $1/\delta(v_1),\dots,1/\delta(v_m)$, respectively, and setting $v_0=0$.
\end{proof}

\begin{remark}\label{rem_algebra}
(a) The number $k$ of Proposition \ref{nondegerate} is called the \emph{grading denominator} by Normaliz, for good reason as the proof shows. The user can choose whether  Normaliz should compute $\Vol(P)$ or $\Vol(kP)$, together with the corresponding Euclidean volume.

(b) Suppose that $\aff(P)$ contains a lattice point. Then the Ehrhart function, the lattice point enumerator of $iP$, $i\in \NN$, is a quasipolynomial $q(i)$ of degree $d=\dim P$,  with constant leading  coefficient $\Vol(P)/d!$. If the grading denominator is $k>1$, then the components $q^{(j)}$ of the quasipolynomial are zero for $j\not\equiv 0 \pod k$, but the components  $q^{(j)}$ for $j\equiv 0 \pod k$  again have constant leading coefficient $\Vol(P)/d!$. By the similarity with (or interpretation as) a Hilbert function it is justified to call $\Vol(P)$ the \emph{multiplicity} of $P$. This extends the standard usage of ``multiplicity'' for lattice polytopes (see \cite[Section 6.E]{BG}).

(c) Normaliz contains all the linear algebra over $\ZZ$ that is necessary for the computations of Proposition \ref{arithmetic}. There are  however aspects that deserve mentioning. For the linear algebra operations mentioned below, Normaliz uses the transformation of integer matrices to row echelon form by the Gaussian algorithm over $\ZZ$, which in its turn is based on the Euclidean algorithm for gcd computations. The determinant of a full rank square matrix is then the product of the diagonal elements of its rowechelon transform.

The basis $u_1,\dots,u_{d+1}$ in  Proposition \ref{arithmetic}(2) is obtained by saturating the sublattice of $\ZZ^n$ that is generated by $v_1,\dots,v_m$. Since $m$ can be extremely large, it saves a substantial amount of time to compute the saturation from a small subset that generates a sublattice of the same rank. Therefore Normaliz tries a random selection that is increased if the rank should not yet suffice.

The saturation itself is computed by a twofold orthogonalization. Let $L$ be a sublattice of $\ZZ^n$. Then we compute the orthogonal sublattice
$$
L^\perp=\{x\in\ZZ^n: \langle x,y \rangle=0 \text{ for all } y\in L  \}.
$$
The saturation of $L$ is given by $L^{\perp\perp}$, as it not hard to see. The computation of $L^\perp$ amounts to solving a homogeneous system of linear equations, and in its turn this task is reduced to the computation of an Hermite normal form \cite[2.4.3]{Cohen}.

The computation of the primitive linear form $\lambda$ amounts to solving a homogeneous system of linear equations as well.

It is difficult to control the size of intermediate results in linear algebra over $\ZZ$. Normaliz uses a twofold strategy to deal with overflows if the user tries a computation with 64 bit integers. The linear algebra operations are constantly monitored for overflows, and if such an overflow occurs for an intermediate result, the whole computation, for example the computation of an Hermite normal form, is repeated in GMP arithmetic. If an overflow occurs in a final result, then Normaliz starts from scratch in GMP arithmetic. Of the examples considered in Section \ref{samples} only one (par-24) uses GMP arithmetic.

The rational numbers that are computed as the volumes of rational polytopes often have very large numerators and denominators. They must not be interpreted as consequences of overflows of 64 bit integers in the linear algebra operations. They are inevitable because we must divide by the products of the degrees of the generators in Proposition \ref{arithmetic}(1) and (3) and further increase by the addition of these fractions.

\end{remark}

\section{Descent sytems}\label{system}

The discussion in the previous section has made it clear that we should compute $\Vol(\overline P)$, given a rational polytope in homogenized coordinates. This is automatically taken care of by the use of Proposition \ref{arithmetic}. All faces of $\overline P$ that come up in the recursive use of Equation \eqref{Vol_dec} are of type $\overline F$ where $F$ is a face of $P$. Therefore we can work directly with $P$ in the combinatorial description of the implementation. To simplify notation in this section, we assume that $\delta(v_i)=1$ for all $i$ (with the notation of Proposition \ref{arithmetic}). Otherwise $\Ht_{\dots}(v_i)$ must be divided by $\delta(v_i)$ whenever it appears. Moreover, we identify $\Vol(P)$ and $\Vol(\overline P)$.

As an example let us discuss a $3$-dimensional cube (or any polytope that is combinatorially equivalent to it).
\begin{figure}[hbt]
	\begin{center}
		\begin{tikzpicture}[y  = {(-0.4cm,-0.5cm)},
		z  = {(0.9659cm,-0.25882cm)},
		x  = {(0cm,0.8cm)},
		scale = 2]
		\draw (0,0,0) -- (1,0,0) -- (1,1,0) -- (0,1,0) --cycle node at (-0.1,1.1,-0.1){$v$};
		\draw (0,0,1) -- (1,0,1) -- (1,1,1) -- (0,1,1) --cycle;
		\draw (0,0,0) -- (1,0,0) -- (1,0,1) -- (0,0,1) --cycle node at (1.1,-0.1,1.1){$w$};
		\draw (0,1,0) -- (1,1,0) -- (1,1,1) -- (0,1,1) --cycle;
		\draw[thick, color=red] (1,0,1) -- (1,1,1);
		\draw[thick, color=red] (1,0,1) -- (1,0,0);
		\draw[thick, color=red] (1,0,1) -- (0,0,1);
		\draw node at (1.1,1.1,-0.1){$v_T$};
		\draw node at (-0.1,1.1,1.1){$v_R$};
		\draw node at (0.0,0.2,0.2){$v_B$};
		\end{tikzpicture}
		\caption{A $3$-dimensional cube}\label{cube}
			
	\end{center}
\end{figure}
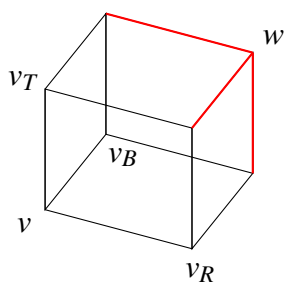
All vertices of $P$ itself have the same number of opposite facets. Let us choose $v$ as in Figure \ref{cube}. Then the opposite facets of $v$ are $T$, $B$ and $R$, namely the top, the back and the right side of $P$. Thus
$$
\Vol(P)=\Ht_T(v)\Vol(T)+\Ht_B(v)\Vol(B)+\Ht_R(v)\Vol(R).
$$
A moment of thought shows that the best choice of the vertices of $T$, $B$ and $R$ are $v_T$, $v_B$ and $v_R$ respectively: only $3$ edges appear as opposite facets, namely the $3$ edges emanating from the vertex $w$ that is antipodal to $v$. The edges are simplices so that we can compute their volumes as determinants, and no further recursion is necessary. Altogether we have constructed a $3$-level system $\cD=(\cD_0,\cD_1,\cD_2)$
$$
\cD_0=\{P\},\ \cD_1=\{T,B,R\},\  \cD_2=\{ T\cap B, T\cap R, B\cap R\}
$$
of faces and distinguished vertices $v,v_T,v_B,v_R$ in the nonsimplex faces that allow the recursive computation of $\Vol(P)$.

\begin{definition}
A \emph{descent system} $\cD=(\cD_i: i=0,\dots,d-1)$  for a polytope $P$ of dimension $d$ is a family of sets $\cD_i$ of faces $F$ together with a map $F\mapsto v(F)$ assigning a vertex $v(F)\in F$ to every nonsimplex $F\in\bigcup_i \cD_i$ such that the following conditions are satisfied for all $i$:
\begin{enumerate}
\item  every $F\in\cD_i$ is a $(d-i)$-dimensional face of $P$;
\item if $G$ is a facet of the nonsimplex face $F\in\cD_{i}$ and $v(F)\notin G$, then $G\in\cD_{i+1}$;
\item if $G\in\cD_{i+1}$, then there exists $F\in \cD_i$ such that $G$ is a facet of $F$ not containing $v(F)$.
\end{enumerate}
(There is no need to introduce $\cD_d$ since all edges of $P$ are simplices.)
\end{definition}

It is immediately clear that a memoryless depth-first recursion would be a bad choice: it does not take into account that lower dimensional faces appear in a large number of higher dimensional ones, and would therefore be computed over and over again. (Compare the numbers $\#\cD$ and $\#\Sigma_{\cF}$ in Table \ref{numerical_data}.)

We compute the descent system by generation:  $\cD_{i+1}$ is computed from $\cD_i$, and if $\cD_i=\emptyset$, the computation is complete; otherwise $\cD_i$ is processed itself in a \emph{parallelized} loop. It is enough to store only the consecutive layers $\cD_i$ and $\cD_{i+1}$ at any time since the recursive application of Equation \eqref{Vol_dec} can be replaced by a forward transfer of the accumulated height information by means of a weight $w(F)$ that is assigned to each face in the descent system. The weight of $P$ is $1$, and the total volume $\Vol(P)$, initially set to $0$, is accumulated step by step. For each face $F\in\cD_{i}$ we perform the following operations:
\begin{enumerate}
\item Decide whether $F$ is a simplex; if so, $w(F)\Vol(F)$ is added to $\Vol(P)$, and we are done with $F$.
\item Otherwise we must find the facets $G$ of $F$,
\item select the vertex $v=v(F)$ (according to a rule explained below),
\item for each facet $G$ of $F$ not containing $v$
\begin{enumerate}
\item compute $\Ht_G(v)$,
\item insert $G$ with $w(G)=0$ into $\cD_{i+1}$ if it has not yet been found by an already processed face $F'\in\cD_i$,
\item increase $w(G)$ by $w(F)\Ht_G(v)$.
\end{enumerate}
\end{enumerate}
The implementation deviates from this description in the treatment of simplex faces; see Remark \ref{modi}(e); however, the algorithm as described makes the theoretical analysis easier.

\begin{proposition}
The algorithm computes $\Vol(P)$ correctly.
\end{proposition}

\begin{proof}
The only question could be whether the weight $w(F)$ is computed correctly. Let us say that the sequence $\cF=(F_0,,\dots,F_k)$ is a \emph{flag} in $\cD$ if $F_i\in \cD_i$ for all $i$ (and therefore $F_0=P$) and $F_{i+1}$ is a face of $F_i$ not containing $v(F_i)$.

Let $F\in\cD_k$ be a face of $P$. It follows from Equation \eqref{Vol_dec} that $\Vol(F)$ contributes to the total volume $\Vol(P)$ with the weight
$$
\sum_{\cF } \Ht_{F_{1}}(v_0)\cdots \Ht_{F_k}(v_{k-1})
$$
where $\cF$ runs through all flags $(F_0,,\dots,F_{k-1},F)$ with  $v_i=v(F_i)$, $i=0,\dots,k-1$. This weight is exactly $w(F)$ as computed by the algorithm. The proposition follows immediately by induction on $k$.
\end{proof}

Polytopes $P$ are usually given as the convex hull of their vertices (V-description) or as the intersection of halfspaces (H-description), where the hyperplanes bounding the halfspaces define the facets of $P$. It is clear from Equation \eqref{Vol_dec} that we need both descriptions. Regardless of which of them defines $P$, one must compute the other one. This is covered by the basic functionality of Normaliz (and many other packages). Once facets and vertices are known, one can compute the incidence matrix  of facets and vertices. It is the basis of all combinatorial computations in the face lattice of $P$.

Since the number of faces in the descent system is potentially very large, the combinatorial details of the implementation are critical. We have tried to find a balance between computation time and memory usage. The main question is what to use as the signature of a face $F$ in the descent system. Since the descent algorithm is meant for polytopes with a moderate number of facets and potentially many vertices, we use the set of facets of $P$ that contain $F$ and not the vertices of $F$. The set of facets is represented by a bitset.

For the ``local'' computations within $F\in\cD_i$ the faces of $F$ are identified by their vertices. These local computations consist of several steps: (i) selecting the vertices of $F$ as those vertices of $P$ that belong to all facets of $P$ containing $F$, (ii) finding the facets of $F$ by intersecting $F$ with the facets of $P$ not containing $F$, (iii) selecting the vertex $v(F)$, computing the heights of $v(F)$ over the facets of $F$ not containing it, and finally (iv) pushing these facets, heights and $w(F)$ to $\cD_{i+1}$. We describe step (ii) in more detail in Remark \ref{modi}(a) below.

Among all candidates for $v(F)$ we choose a vertex $v$ of $F$ that (i) minimizes the set of ``opposite'' facets of $F$, and (ii) then minimizes the number of faces  $F'\in \cD_i$ containing $v$. While rule (i) is an obvious choice, rule (ii) tries to take $v$ as ``exterior'' as possible in the set $\bigcup_{F'\in\cD_i} F'$, so that the facets sent to $\cD_{i+1}$ share as many subfacets as possible. (The choice of $v_T,v_T,v_R$ for the cube illustrates this rule.) The introduction of rule (ii) has reduced the size of the descent systems typically by $20\%$.

\begin{remark}\label{modi}
(a) Every face of a polytope $P$ is the intersection of a set of support hyperplanes $H$ of $P$ with $P$. This implies for a given face $F$ that each facet $F'$ of $F$ is obtained as the intersection $F'=F\cap  H'$ with a support hyperplane $H'$ of $P$, $F\not\subset H'$. After the vertex set of $F$ has been computed as the intersection of all $H\supset F$, we compute the intersections $F\cap H$, $F\not\subset H$. This operation yields a set of faces $\cF$ of $F$. In general, $\cF$ contains also non-facets of $F$, and a face $F''\in\cF$ can be cut out by several hyperplanes $H$.

Despite of these observations, finding the facets $F'$ of $F$ in $\cF$ is a purely set theoretic task: we must find the maximal elements in $\cF$. These are exactly the facets of $F$, and each such facet $F'$ satisfies $\dim F'=\dim F-1$.

All the operations just mentioned use only the facet-vertex incidence vectors of $P$. There is no need to compute dimensions of faces which would be an alternative for selecting the facets among the elements of $\cF$.

(b) Normaliz tests whether the polytope $P$ is simple. For simple polytopes the situation is simpler (!): If $F$ is a face of $P$ and $H\not\supset F $ a support hyperplane of $P$, then either $F\cap H=\emptyset$ or $G=F\cap H$ is a facet of $F$. Moreover, $H$ is the only support hyperplane of $P$ that cuts out $G$ from $F$.

(c) The vertex sets of $G$ are known for the facets $G$ of $F\in\cD_i$ that go into $\cD_{i+1}$. Storing them with $G$ would accelerate the computation somewhat, but would require considerably more memory, making computations for polytopes with large vertex sets impossible.

(d) One could modify rule (i) for the selection of $v(F)$ by counting only nonsimplex facets that contain $v$. Experiments have shown that this is not a good choice.

(e) Instead of sending simplex facets of $F\in \cD_i$ into $\cD_{i+1}$ the implementation computes them directly. This has almost no influence on computation time in general, but reduces memory usage somewhat. For simplicial polytopes the gain is however tremendous.

(f) If the number of faces in $\cD_i$ exceeds one million, $\cD_i$ is processed in blocks of this size, and each block is freed when it is finished. This reduces memory usage further.
\end{remark}

\begin{remark}\label{complexity}
We give an overview of the complexity of the descent algorithm. It is proportional to the total number $\#\cD=\sum_i \#\cD_i$. With $h$ denoting the number of facets of $P$ and $V$ the number of its vertices, the bit operations \emph{per face} $F$ can be estimated as follows:
\begin{enumerate}

\item $O(hV)$ bit operations for finding the vertices of $F$, the candidates for the facets of $F$, and selecting the vertex $v(F)$,

\item if $P$ is not simple, $O(h^2V)$ bit operations for selecting the facets among the faces found in (2),

\item $O(h^2\log \#\cD)$ bit operations for the insertion of the facets of $F\in\cD_i$ into $\cD_{i+1}$.
\end{enumerate}

These are rough estimates that do not take into account that many bit operations are implemented as operations of bit sets represented by a vector of words of size $64$.

The  linear algebra methods over $\ZZ$ are described in Remark \ref{rem_algebra}. Their complexity is very difficult to estimate. Even authoritative sources such as Cohen \cite{Cohen} do not contain bounds.
\end{remark}

The role of the simplices in the descent system raises the suspicion that the algorithm implicitly uses a triangulation or at least a decomposition with similar properties. This is indeed the case. For each complete flag $\cF=(F_0,\dots,F_k)$, in which $F_k$ is necessarily a simplex, set
$$
\Sigma_\cF=\conv(v_0,\dots,v_{k-1},F_k),\qquad v_i=v(F_i).
$$
By the choice of $v_0,\dots,v_{k-1}$ and $F_k$ this set is indeed a $d$-simplex, and one has
$$
P=\bigcup_{\cF} \Sigma_\cF.
$$
Moreover, the relativ) interiors of the simplices $\Sigma_\cF$ are pairwise disjoint, and this property is good enough for volume computations. In general $\Sigma_\cF\cap \Sigma_{\cF'}$ is not a face of both simplices so that the decomposition is not a triangulation in the strong combinatorial sense. If a true triangulation is desired, one has to fix an order of the vertices $v_1,\dots,v_m$ of $P$ beforehand, and for every nonsimplex face $F$ select $v(F)$ as the first vertex that belongs to $F$. The triangulation constructed in this way is reverse-lexicographic in the sense of the Sturmfels correspondence or pulling in combinatorial terminology (for example, see \cite[Section 7.C]{BG}) .

The primal algorithm of Normaliz that builds a cone (over a polytope) incrementally by successively adding generators produces a lexicographic (or placing) triangulation. Its construction is discussed in \cite{BIS}. Lexicographic triangulations have many advantages and go very well with the Fourier-Motzkin elimination in convex hull computations.

However, if a cone or polytope is given by inequalities, then the reverse-lexicographic approach is more natural. Future versions of Normaliz may use it as well for the computations of triangulations. Nevertheless note that its success in volume computation is based on the fact that the number $w(F)$ captures the relevant information of the set of all flags ending in $F$. We will illustrate this effect by several sample computations in the next section.

\section{Sample computations}\label{samples}

\subsection{The test polytopes}
We demonstrate the power of the descent algorithm by some sample calculations. The following polytopes are used:

\begin{enumerate}
\item \emph{\ttt{Strict Borda}} is the polytope underlying the computation of the probability of the strict Borda paradox in social choice; see \cite{BIS2} for the details.

\item \emph{\ttt{Condorcet}} is one of the polytopes that appears in relation with the Condorcet's other paradox. It is discussed in Section \ref{4_rules}, where it is labeled as $\cQ_1$.

\item \emph{\ttt{4 rules}} comes from social choice as well. Again, it is discussed in Section \ref{4_rules}.

\item \emph{\ttt{8x8-score}} represents the monoid
    of ``$8\times 8$ ordered score sheets'' and was discussed in \cite{IM} .

\item \emph{\ttt{6x6-magic}} represents the monoid
    of ``$6\times 6$ magic squares'', that is the monoid of squares of size $6\times 6$ filled with
    nonnegative integers such that all rows,
    columns and the two main diagonals have the same sum called the ``magic
    constant''.

\item \emph{\ttt{d-par}} is a parallelotope of dimension $d$ produced as a test example.

\item \emph{\ttt{d-cube}} is the unit cube of dimension $d$.

\item \emph{\ttt{bool mod $S_5$}} represents the boolean model for
    the symmetric group $S_5$ and \emph{\ttt{lin ord $S_6$}} is the linear order polytope for the symmetric group $S_6$; they belong to the area of
    statistical ranking, see \cite{StW} for example.

\item \emph{\ttt{A443}} and \emph{\ttt{A543}} are monoids defined by the
    $2$-dimensional  marginal distributions of the $3$-dimensional
    contingency tables of sizes $4\times4\times3$ and
    $5\times4\times3$. In the
    classification of Ohsugi and Hibi \cite{OH}
    are listed as open cases and were
    closed in \cite{BHIKS}.

\item \emph{\ttt{cyclo60}} represents the cyclotomic monoid of
    order $60$ and was discussed by Beck and Ho\c{s}ten in
    \cite{BeH}.

\item \emph{\ttt{d-cross}} is the unit cross polytope of dimension $d$ spanned by the unit vectors $\pm e_i$, $i=1,\dots,d$.
\end{enumerate}

The first $9$ polytopes in Table \ref{numerical_data} are defined by inequalities and equations whereas the other $8$ are lattice polytopes given by their vertices. The computation times in the ``primal'' and ``descent'' columns  of Table \ref{calculation_times} include the conversion from one representation to the other; for those in the ``special'' column it is superfluous.  All computations can be (and were) done in 64 bit arithmetic, with the exception of \ttt{24-par} that needs GMP integers.

\begin{table}[hbt]
\begin{tabular}{|r|r|r|r|r|r|r|}
\hline
\strut                 & edim & $\#$ supp & $\#$ vert  &  $\# \cD$   & $\# \det$      & $\# \Sigma_\cF$ \\
\midrule[1.2pt]
\strut \ttt{strict Borda}   &  24       & 33        & 6,363      & 4,407,824   &   901,955   & $2\cdot 10^{10}$    \\
\hline
\strut \ttt{Condorcet}      &  24       & 33        & 51,168     & 82,524,473  &  22,222,231 & $2.7\cdot 10^{12}$  \\
\hline
\strut \ttt{4 rules}        &  24       & 36        & 233,644    & 652,216,133 & 177,513,245 & $17.9\cdot 10^{12}$  \\
\hline
\strut \ttt{8x8-score}      &  56       & 63        & 6,725,600  & 6,725,550   & 343         & $4.2\cdot 10^{40}$  \\
\hline
\strut \ttt{6x6-magic}      &  24       & 36        & 97,548     & 494,867,792 & 113,068,158 & $31.8\cdot 10^{12}$ \\
\hline
\strut \ttt{20-par}         &  21       & 40        & $2^{20}$   & $2^{20}-21$ & 380         & $20!$ \\
\hline
\strut \ttt{24-par}         &  25       & 48        & $2^{24}$   & $2^{24}-25$ & 552         & 24! \\
\hline
\strut \ttt{20-cube}        &  21       & 40        & $2^{20}$   & $2^{20}-21$ & 380         & $20!$ \\
\hline
\strut \ttt{24-cube}        &  25       & 48        & $2^{24}$   & $2^{24}-25$ & 552         & $24!$\\
\midrule[1.2pt]
\strut \ttt{bool mod $S_5$} &  27       & 235       & 120        & 14,541,872  & 334,154     & $2\cdot 10^{10}$  \\
\hline
\strut \ttt{lin ord $S_6$}  &  16       & 910       & 720        & 19,012,391  & 2,133,900   & $5.8\cdot 10^9$ \\
\hline
\strut \ttt{A443}           &  30       & 4948      & 48         & 204,363     & 22,334      & 2,654,224 \\
\hline
\strut \ttt{A543}           &  36       & 29387     & 60         & 3,049,328   & 183,519     & $10^8$ \\
\hline
\strut \ttt{cyclo60}        &  17       & 656100    & 60         & 1,712,752   & 149,253     & 9,188,100 \\
\hline
\strut \ttt{20-cross}       &  21       & $2^{20}$  & 40         & $1$         & $2^{19}$    & $2^{19}$ \\
\hline
\strut \ttt{24-cross}       &  25       & $2^{24}$  & 48         & $1$         & $2^{23}$    & $2^{23}$ \\
\hline
\strut \ttt{28-cross}       &  29       & $2^{28}$  & 56         & $1$         & $2^{27}$    & $2^{27}$ \\
\hline
\end{tabular}
\vspace*{1ex}
\caption{Numerical data} \label{numerical_data}
\end{table}

\begin{table}[hbt]
\begin{tabular}{|r|r|r|r|r|r|r|}
\hline
\strut                 & \multicolumn{3}{c|}{RAM in GB} & \multicolumn{3}{c|}{time }   \\
\cline{2-7}
\strut                 &  \multicolumn{1}{c|}{primal}&  \multicolumn{1}{c|}{special}  &  \multicolumn{1}{c|}{descent} & \multicolumn{1}{c|}{primal}&  \multicolumn{1}{c|}{special} &  \multicolumn{1}{c|}{descent} \\
\midrule[1.2pt]
\strut \ttt{strict Borda}   &  1.03   &         & 0.36              &  5:30:27 h  &           & 25.37 s     \\
\hline
\strut \ttt{Condorcet}      &         &         & 4.26              &             &           & 14:42 m  \\
\hline
\strut \ttt{4 rules}        &         &         & 34.39             &             &           & 4:50:52 h \\
\hline
\strut \ttt{8x8-score}      &         & 9.92    & 12.78             &             & 6:02 m    & 4:38:12 h \\
\hline
\strut \ttt{6x6-magic}      &         &         & 22.77             &             &           & 1:59:43 h \\
\hline
\strut \ttt{20-par}         &         & $<0.01$ & 1.11              &             & 0.06 s    & 2:35 m\\
\hline
\strut \ttt{24-par}         &         & $<0.01$ & 66.67             &             & 0.10  s   & 12:39:49 h \\
\hline
\strut \ttt{20-cube}        &         & $<0.01$ & 1.11              &             & $<0.01$ s & 2:26 m \\
\hline
\strut \ttt{24-cube}        &         & $<0.01$ & 20.94             &             & $<0.01$ s & 11:57:01 h \\
\midrule[1.2pt]
\strut \ttt{bool mod $S_5$} & 1.10    &         & 0.69              &  1:10:07 h  &           & 2:28 m  \\
\hline
\strut \ttt{lin ord $S_6$}  & 0.84    &         & 1.96              &  19:03 m    &           & 2:22 m \\
\hline
\strut \ttt{A443}           & 0.68    &         & 0.05              &   1.55 s    &           & 18.35 s  \\
\hline
\strut \ttt{A543}           & 0.94    &         & 1.83              &  30.06 s    &           & 26:26 m \\
\hline
\strut \ttt{cyclo60}        & 1.30    &         & 96.37             &  40.2 s     &           & 3:13:57 h \\
\hline
\strut \ttt{20-cross}       & 1.96    &         & 1.07              &  12.90 s    &           &  11.79 s    \\
\hline
\strut \ttt{24-cross}       & 14.79   &         & 22.87             &  2:20 m     &           & 3:56 m  \\
\hline
\strut \ttt{28-cross}       & 203.43  &         & 354.47            &  57:07 m    &           & 1:47:08 h  \\
\hline

\end{tabular}
\vspace*{1ex}
\caption{Memory usage and times for parallelized volume calculations} \label{calculation_times}
\end{table}

In Table \ref{numerical_data} $\textup{edim}$ is the dimension of the space in which the polytope is computed -- it is $\dim P+1$. The number of vertices is denoted by $\# \textup{vert}$  and that of support hyperplanes by $\# \textup{supp}$. Moreover, $\# \cD$ is the total size of the descent system, $\#\det$ the number of determinants computed by the descent algorithm, and $\# \Sigma_\cF$ the number of simplices in a decomposition that would be produced by the algorithm. (The triangulations used by the primal algorithm usually have somewhat different sizes, and also the number of computed determinants is most often different.)

\subsection{Parallelized volume computations} The computation times in Table \ref{calculation_times} are ``wall clock times'' taken on a Dell R640 system with two Intel\texttrademark Xeon\texttrademark Gold 6152 (a total of $44$ cores)   using $20$ parallel threads (of the maximum of $88$). The efficiency of the parallelization is discussed below. We define it as the quotient
$$
\frac{T_1}{tT_t}
$$
where $T_1$ is the time of the strictly sequential computation, $t$ the number of threads and $T_t$ the time of parallel computation with $t$ threads.
The times listed are for the descent algorithm discussed in this paper, the Normaliz primal algorithm using lexicographic triangulations, and special algorithms that can be applied in some cases; see Remark \ref{comp_remarks}(b) and (e).

For the primal algorithm a missing entry in Table \ref{calculation_times} means that the computation is not doable in a reasonable amount of time since the triangulation would have $>10^{12}$ simplices. A missing entry in the `special'' column indicates that Normaliz has no special method for the particular example.

\begin{remark}\label{comp_remarks}
(a) A profiler run of the example \ttt{Strict Borda}, which we consider as a typical application of the descent algorithm, shows that $\approx 43\%$ of the computation time are spent on linear algebra, whereas the bitset operations take $\approx 26 \%$. The rest goes into preparations and administration.

(b) Among the polytopes calculated by the primal algorithm, \ttt{strict Borda} is by no means the biggest (see \cite{BIS} for much larger computations). However, among the polytopes calculated for \cite{BIS2} it is the largest since most others can be simplified by symmetrization (see \cite{BS}). Symmetrization can be applied very efficiently to \ttt{8x8-score}, and this is the special algorithm used for it in addition to descent.  For it, the triangulation is approximately $10^{34}$ times as large as the descent system.

(c) Despite of the special algorithm for parallelotopes described in (e), we have run the descent algorithms on some of them since the results are predictable and can therefore be used as tests for correctness. Moreover, they are prototypes of simple polytopes with very few facets, but a large number of vertices.

The polytope \ttt{20-par} is an affine image of the \ttt{20-cube}. It is not hard to see that the selection rule for vertices in non-simplex faces produces the descent system $\cD$ consisting exactly of the faces containing the vertex $w$ antipodal to the vertex $v(P)$, as illustrated for the 3-cube by Figure \ref{cube}. The simplex faces are the lines emanating from $w$. In this case the algorithm implicitly produces an affine image of the Knudsen-Mumford triangulation determined by the root system $A_{20}$ (for example, see \cite[Section 3.A]{BG}). Parallelotopes profit from the special handling of simple polytopes; see Remark \ref{modi}(b).

(d) Analogous remarks apply to \ttt{24-cube} and \ttt{24-par}. If one compares the computation times of \ttt{24-cube} with its trivial arithmetic to that of \ttt{24-par} with substantially more complicated arithmetic, it becomes clear that the bulk of the computation time for these polytopes goes into the (identical) combinatorics. The handling of the very long bitsets representing the vertices in a face is the bottleneck in these computations, as becomes apparent also from \ttt{8x8-score}.

(e) As the column ``special'' shows, there is a tremendously faster approach to the parallelotopes: if $P$ is a $d$-parallelotope, then $\Vol(P)=d!\Vol(\sigma)$ where $\sigma$ is a ``corner'' simplex of $P$ spanned by a vertex and its neighbors. If there should be any doubt: this follows from the transformation rule for volumes, once it has been observed for the unit cube.

The recognition of parallelotopes was added to Normaliz for the computation of lattice points in such polytopes, as they appear in numerical mathematics; see  Kacwin, Oetters\-hagen and Ullrich \cite{KOU}. One could of course add a recognizer for cross $d$-polytopes as well. Again a single simplex would be sufficient: $\Vol(P)=2^{d-1}\Vol(\sigma)$ for every simplex $\sigma$ spanned by a vertex and an opposite facet. (It is enough to consider the unit cross polytope.)

(f) For the polytopes with a huge number of facets or vertices the transfer of these data between the components of the Normaliz system of course takes its toll.

(g) For the polytopes defined by vertices the primal algorithm is usually more efficient, as shown by several of the last $8$ polytopes. This was to be expected for those with a large number of non-simplex facets, but for the simplicial cross polytopes the difference is small.

The number of facets is moderate for \ttt{bool mod $S_5$} and \ttt{lin ord $S_6$}; nevertheless it came as a surprise that the descent algorithm is significantly faster than the primal algorithm.

The cross polytopes are a class for which exact computation seems to be faster than probabilistic methods. The computation time for \ttt{18-cross} reported in \cite[Table 1]{Emiris} is much higher than ours for \ttt{20-cross}. This is of course also true for parallelotopes if one uses the special approach explained in (e) above.
\end{remark}

Table \ref{para} documents the efficiency of parallelization on two different systems, the Dell R640 mentioned above and another system equipped with $2$ Intel\texttrademark Xeon\texttrademark  E5-2660 at 2.20GHz (a total of 16 cores and 32 threads). The test example is \ttt{Condorcet}. The computation times for a single thread are $213$ minutes on the R640 and $370$ minutes on system 2. Until $16$ threads the efficiency of parallelization on both systems is almost equal and very acceptable.
\begin{table}[hbt]
\begin{tabular}{|c|r|r|r|r|r|r|}
\hline
\strut $\#$ threads & 1 & 2 & 4 & 8 & 16 & 32 \\
\hline
\strut R640      & 100 & 94 & 89 & 84 & 77 & 53 \\
\hline
\strut System 2 & 100 & 98 & 98 & 94 & 81 & 43 \\
\hline
\end{tabular}
\vspace*{1ex}
\caption{Efficiency of parallelization in \%}\label{para}
\end{table}

\subsection{Normaliz vs. vinci}

\begin{table}[hbt]
\scalebox{0.85}{
	\begin{tabular}{|r|r|r|r|r|r|r|r|}
		\hline
		\strut                 & \multicolumn{3}{c|}{Normaliz 1x} & \multicolumn{4}{c|}{Vinci }   \\
		\cline{2-8}
		\strut                 &  \multicolumn{1}{c|}{-s} &  \multicolumn{1}{c|}{primal}  &  \multicolumn{1}{c|}{descent} & \multicolumn{1}{c|}{cdd or \textit{lrs}} & \multicolumn{1}{c|}{rlass}&  \multicolumn{1}{c|}{hot} &  \multicolumn{1}{c|}{rch} \\
		\midrule[1.2pt]
		\strut \ttt{strict Borda}   & 0.99 s  &  50:15:12 h     & 4:33 m        & 2.38 s      &  1:25 m     & 1:18 m      & 28:32:03 h \\
		\hline
		\strut \ttt{Condorcet}      & 1.84 s  &  T              & 4:18:56 h     & 26.83 s     &  R          & 50:05 m     & T \\
		\hline
		\strut \ttt{4 rules}        & 6.76 s  &  T              & 85:38:56 h    & 9:49 m      &  R          & 7:54:26 h   & T \\
		\hline
		\strut \ttt{8x8-score}      & 5:48 m  &  T              & 95:09:03 h    & \textit{9:54} m      &  E          & T           & T \\
		\hline
		\strut \ttt{6x6-magic}      & 6.13 s  &  T              & 31:05:16 h    & 4:07 m      &  E          & LD          & T \\
		\hline
		\strut \ttt{20-par}         & 12.91 s &  T              & 44:29 m       & 2:26:38 h   &  1:53 m     & R           & T \\
		\hline
		\strut \ttt{24-par}         & 11:33 s &  T              & 216:09:45 h   & \textit{6:30} m      &  R          & T           & T \\
		\hline
		\strut \ttt{20-cube}        & 7.29 s  &  T              & 42:55 m       & 3:20:38 h   &  33.47 s    & R           & T \\
		\hline
		\strut \ttt{24-cube}        & 2:40 m  &  T              & 169:47:10 h   & \textit{6:38} m      &  R          & T           & T \\
		\midrule[1.2pt]
		\strut \ttt{bool mod $S_5$} & 0.18 s  & 5:50:36 h       & 32:57 m       & 1.69 s      &  T          & LD          & LD   \\
		\hline
		\strut \ttt{lin ord $S_6$}  & 30.98 s & 1:48:36 h       & 40:05 m       & 2:33 m      &  H          & 5:15 m      & 59:31:08 h \\
		\hline
		\strut \ttt{A443}           & 0.46 s  & 7.16 s          & 4:10 m        & 1.00 s      &  H          & LD          & LD   \\
		\hline
		\strut \ttt{A543}           & 13.65 s & 4:55 m          & 12:14:55 h    & 2:20 m      &  H          & LD          & LD   \\
		\hline
		\strut \ttt{cyclo60}        & 1:32 m  & 2:20 m          & 89:09:53 h    & 35:21:10 h  &  H          & 44:05:45 h  & 85:57:56 h \\
		\hline
		\strut \ttt{20-cross}       & 5.21 s  & 9.80 s          & 18.06 s       & 2:15:07 h    &  H          & 1:22:42 h   & 1:07:12 h \\
		\hline
		\strut \ttt{24-cross}       & 1:57 m  & 2:22 m          & 5:53 m        & \textit{4:24} m      &  H          & T           & T  \\
		\hline
		\strut \ttt{28-cross}       & 54:34 m & 1:22:34 h       & 2:28:17 h     & \textit{1:36:45} h   &  H          & T           & T  \\
		\hline
		
	\end{tabular}}
	\vspace*{1ex}
	\caption{Computation times Normaliz (single thread) vs. vinci} \label{Benchmark_times}
\end{table}

Tables \ref{Benchmark_times} and \ref{Benchmark_memory} compare Normaliz, run with a single thread, to vinci \cite{vinci}, \cite{practical}. Among the algorithms offered by vinci we took the three that do not need third party software (except cdd) and do not a priori restrict the class of polytopes to which they can be applied. These are
\begin{enumerate}
\item rlass, a revised version of Lasserre's algorithm \cite{lass};
\item hot, ``hybrid orthonormalization technique'', a recursive algorithm that is very similar to Normaliz' descent;
\item rch, a revised version of Cohen and Hickey's combinatorial triangulation algorithm \cite{coh_hick}.
\end{enumerate}

\begin{table}[hbt]
	\begin{tabular}{|r|r|r|r|r|r|}
		\hline
		\strut                 & \multicolumn{2}{c|}{Normaliz 1x} & \multicolumn{3}{c|}{Vinci }   \\
		\cline{2-6}
		\strut                 &  \multicolumn{1}{c|}{primal}  &  \multicolumn{1}{c|}{descent} & \multicolumn{1}{c|}{rlass}&  \multicolumn{1}{c|}{hot} &  \multicolumn{1}{c|}{rch} \\
		\midrule[1.2pt]
		\strut \ttt{strict Borda}   &  0.75           & 0.3             &  58.3       & 1.37      & 0.003  \\
		\hline
		\strut \ttt{Condorcet}      &  T              & 4.32            &  R          & 42.5      & T  \\
		\hline
		\strut \ttt{4 rules}        &  T              & 35.42           &  R          & 317.26    & T \\
		\hline
		\strut \ttt{8x8-score}      &  T              & 14.27           &  E          & T         & T \\
		\hline
		\strut \ttt{6x6-magic}      &  T              & 23.06           &  E          & LD        & T \\
		\hline
		\strut \ttt{20-par}         &  T              & 1.06            &  72.14      & R         & T \\
		\hline
		\strut \ttt{24-par}         &  T              & 84.54           &  R          & T         & T  \\
		\hline
		\strut \ttt{20-cube}        &  T              & 1.06            &  26.94      & R         & T \\
		\hline
		\strut \ttt{24-cube}        &  T              & 18.79           &  R          & T         & T \\
		\midrule[1.2pt]
		\strut \ttt{bool mod $S_5$} & 0.79            & 0.66            &  T          & LD        & LD   \\
		\hline
		\strut \ttt{lin ord $S_6$}  & 0.58            & 1.97            &  H          & 1.85      & 0.001 \\
		\hline
		\strut \ttt{A443}           & 0.50            & 0.4             &  H          & LD        & LD   \\
		\hline
		\strut \ttt{A543}           & 0.72            & 1.79            &  H          & LD        & LD   \\
		\hline
		\strut \ttt{cyclo60}        & 0.77            & 66.05           &  H          & 1.19      & 0.2 \\
		\hline
		\strut \ttt{20-cross}       & 0.26            & 0.93            &  H          & 0.32      & 0.32 \\
		\hline
		\strut \ttt{24-cross}       & 9.71            & 15.94           &  H          & T         & T  \\
		\hline
		\strut \ttt{28-cross}       & 159.39          & 271.76          &  H          & T         & T  \\
		\hline
		
	\end{tabular}
	\vspace*{1ex}
	\caption{Memory usage Normaliz (single thread) vs. vinci} \label{Benchmark_memory}
\end{table}

\begin{remark}\label{vinci_rem}
(a) As pointed out in the introduction, vinci computes only Euclidean volumes. Floating point computations may be sufficient for applications as in Section \ref{Volumes Computations}, and, as a posteriori comparisons with the exact  rational computations of Normaliz show, the approximations computed by vinci are very good. However, we do not know of any a priori error bounds for volume computations in floating point.

(b) Not all computations were successful. The letter T in the tables indicates that the computation time exceeds 250 h. The letter R means that the RAM usage exceeds 500 GB. Moreover, LD stands for low dimension (see
(d)), and H indicates that the number of facets exceeds the preset bound of 254 for rlass. The letter E indicates that vinci ended with an error message whose cause
we could not find out.

(c) Convex hull computation and vertex enumeration are not contained in vinci. Instead we used cdd \cite{cdd} for this step as recommended by the vinci documentation, except in five computations where cdd did not finish within $50$ hours. These were done by lrs \cite{lrs}. The times of lrs are marked by italics in the table. The time spent by cdd or lrs has not been added to the vinci computation times, whereas the Normaliz times contain convex hull and vertex enumeration. These times were measured in separate computations and appear in the \texttt{-s} column of Table \ref{Benchmark_times}.

The maximum Normaliz time with a single thread for this step is 54:34 m (\texttt{28-cross}) and the second largest is 5:48 m (\texttt{8x8-score}). Benchmarks for convex hull computations that include cdd and Normaliz  can be found in \cite{pmk-bench} and \cite{Koeppe}. Note that the lrs times have the same order of magnitude as the Normaliz times, whereas the cdd times are often much larger.

(d) We did not succeed to compute the volumes of lower dimensional polytopes, i.e., polytopes whose dimension is smaller than that of the ambient space, with vinci. The volume computed by vinci is then $0$, which is of course correct with respect to the full dimensional volume of the ambient space, but this information is useless. In general, lower dimensional polytopes have no rational, isometric and full dimensional embedding, as the diagonal of the unit square shows.

The polytopes \texttt{strict Borda}, \texttt{Condorcet} and \texttt{4 rules} are lower dimensional, but for them there exists a workaround, for Euclidean as well as for lattice volumes. The latter are needed for applications like those in Section \ref{Volumes Computations}. Let $P$ be one of these polytopes. Then $H=\aff(P)$ is the affine hyperplane through the unit vectors. The lattice height of the origin over this hyperplane is $1$, and the Euclidean height is $1/\sqrt{24}$. Set $\overline P=\conv(P,0)$. It is enough to compute one of the volumes of $\overline P$ since
$$
\vol(\overline P) = \frac{1}{24} \sqrt{24}\vol(P),\qquad \Vol(\overline P)=\Vol(P), \qquad \Vol(\overline P)= 24! \vol(\overline P).
$$
\end{remark}

The descent algorithm of Normaliz uses the same recursive approach as vinci's
hot. But the implementations differ significantly. Normaliz' descent has ben designed for low memory usage, as pointed out in Section \ref{system}. This is a clear advantage for really large face lattices like those of \texttt{Condorcet}, \texttt{4 rules} or the parallelotopes, for which the memory usage differs by a factor of $10$ or more. On the other hand, when there is no shortage of memory, hot is sometimes significantly faster than Normaliz with a single thread. Of course, the floating point arithmetic of vinci is a general advantage for computation time.

As the primal algorithm of Normaliz, rch evaluates a reverse lexicographic triangulation. On the whole, it is significantly slower than Normaliz lexicographic triangulation method. Both algorithms have modest memory requirements.

Normaliz has no equivalent of rlass. Therefore a direct comparison is not possible. While rlass is certainly ultrafast when it can be applied, its greediness for memory sets a rather tight bound for applications. It is our impression that the recursion tree of rlass is growing extremely fast. (For \texttt{bo5} it broke the time barrier without any result, for which we had expected $0$ since the polytope is lower dimensional.)

(e) Several computations in the case of four candidates elections by Diss, Kamwa and Tlidi in \cite{Diss} were done with
Normaliz. We tried to repeat them with vinci, but did not succeed since
both rlass and hot broke the memory barrier.

\section{Application: Computations of volumes in four candidates elections}\label{Volumes Computations}

The appearance of rational polytopes in
social choice is fully discussed in \cite[Section 2]{BIS2} and we use the same notations in the following. We refer the reader to \cite{GL} or \cite{GL2}  for extra details and a more extensive
treatment. The basic assumption is that each voter has a linear preference order of the candidates in an election. If there are $n$ candidates, then the number of preference orders is $N=n!$. The result of the election is the vector $(x_1,\dots,x_N)$ that for every $i$ lists the number of voters having preference order  $i=1,\dots,N$. The further  computations are based on the \emph{Impartial Anonymous Culture} (IAC) assumption, see \cite[Section 2]{BIS2} for details. (IAC) assumes that for a fixed number $k$ of voters all election results have equal probability. This allows the computation of probabilities, as $k\to\infty$, of certain phenomena as lattice normalized volumes of rational polytopes.

\subsection{Four voting rules, same winner}\label{4_rules}
First, we recall four well-known voting rules.
\begin{enumerate}
\item The \emph{plurality rule} ($PR$): the voters cast one vote for their preferred candidate. The \emph{plurality winner} ($PW$) is the candidate which has the most first places in the
preference orders of the voters.

\item The \emph{negative plurality rule} ($NPR$): it requires the voters to cast one vote against their least preferred candidate. The \emph{negative plurality winner} ($NPW$) is the candidate which has the fewest last places in the
preference orders of the voters.

\item The \emph{majority rule} ($MR$): all voters preferences are considered and we say that a candidate $A$ "beats" a candidate $B$ by \emph{pairwise majority rule}
if there are more voters which prefer $A$ to $B$ than voters that prefer $B$ to $A$. The \emph{Condorcet winner} ($CW$), i.e. the majority rule winner,
is the candidate which beats all other candidates by the pairwise majority rule.
As the Marquis de Condorcet \cite{C} observed, the relation "beats" is nontransitive in general, and one must consider the possibility of Condorcet's paradox, namely an outcome  without a Condorcet winner.

\item The \emph{Borda rule} ($BR$): this is a weighted scoring rule which in the particular case of four candidates assigns $3$ points to a candidate for each most-preferred ranking in a
voter's preferences, $2$ points for each second-place ranking, $1$ point for each third-place ranking and zero points for each least-preferred ranking. The \emph{Borda winner} ($BW$) is the candidate which cumulates
the most points.
\end{enumerate}

We want to compute the probability that all four voting rules deliver the same winner in four candidates elections as the numbers of voters $k$ goes to $\infty$.

Let us choose a candidate $A$. The polytope $\cP$ associated to the event that $A$ is the winner of all four voting rules is cut by $36$ inequalities and $1$ equation from $\RR^{24}$: $24$ inequalities $x_i\ge 0$, $3$ inequalities for each of the $4$ rules fixing $A$ as the winner, and the equation $x_1+\dots+x_{24}=1$; see \cite{BIS2} for several related systems of equations.
The combinatorial data of the polytope $\cP$ and the computation time are listed in Table \ref{numerical_data} and Table \ref{calculation_times}, \texttt{4 rules}.

We have obtained
$$
\vol \cP=\frac{a}{b},
$$
where
\begin{align*}
	a= & \scalebox{0.9}{15434295102897069492696787224587569493324878059069286556500157094466280221} \\
	& \scalebox{0.9}{0031839904092203533576766900008697462518883193615863751857064434519917747}
\end{align*}
and
\begin{align*}
	b= & \scalebox{0.9}{19734891994161694286368932836293271062441599301077174316463585667073366250} \\
	& \scalebox{0.9}{92787497360174222493081399494071993084340140223731960203182080000000000000.}
\end{align*}

The probability that all four voting rules deliver the same winner in four candidates elections may then be computed as
$$
4\cdot \vol \cP= \frac{4\cdot a}{b}\approx 0.312833.
$$
We were surprised by this rather small value: even if a Condorcet winner exists, the winner of the actual voting scheme is rather unpredictable. It would of course be possible to analyze the situation further by considering $3$ rules versus the $4$-th in each case. The computations need some hours, but they are well accessible.

\subsection{On Condorcet's other paradox}  In \cite{C2} Condorcet presents several examples of voting paradoxes that may appear in three candidates elections.
In particular we are interested in \cite[Example 4, page 150]{C2}, illustrating a voting situation in which the Condorcet winner is the same as the plurality winner, but not the Borda winner.
We want to compute the probability that this phenomenon will appear in four candidates elections under (IAC).

Set candidate $A$ to be both the plurality and the Condorcet winner.
Since the Borda rule gives a total order of the candidates, we have four situations that may appear:
\begin{enumerate}
	\item $A$ is placed first by the Borda rule. We denote the corresponding polytope by $\cQ_1$;
	\item $A$ is placed second by the Borda rule. We have to make a choice for the winning candidate. Assume that $B$ beats $A$ by the Borda rule and denote the corresponding polytope by $\cQ_2$;
	\item $A$ is placed third by the Borda rule. We have to make a choice for the losing candidate. Assume that $B$ and $C$ beat $A$ by the Borda rule ($D$ is placed on last place) and denote the corresponding polytope by $\cQ_3$;
	\item $A$ is placed last by the Borda rule (or in other words $A$ is the \emph{Borda loser}). We denote the corresponding polytope by $\cQ_4$;
\end{enumerate}

All polytopes  of this family are cut by $33$ inequalities and $1$ equation in dimension $24$. Not all the inequalities are relevant, however minimizing the number of inequalities does not make the problem easier to solve.

A fast computation shows that the polytope $\cQ_4$ has empty absolute interior (i.e., the dimension is $<23$), so its 23-dimensional volume is zero. In fact, it is known in general, see \cite[Theorem 4]{FG},
that the Condorcet winner cannot be the Borda loser. Further, there are two independent ways to compute the probability that the Condorcet winner is the same as the plurality winner, but not the Borda winner.
It can be computed directly, with the formula:
$$
12\cdot (\vol \cQ_2+\vol \cQ_3).
$$
It can also be computed indirectly, using the fact that the probability that the Condorcet winner is the same as the plurality winner was computed previously in \cite{Sch}.
We recall from \cite[Subsection 2.3]{BIS2} that the volume of the polytope associated  is:
$$
\vol \cE=\frac{10658098255011916449318509}{68475651442606080000000000},
$$
and the formula is:
$$
4\cdot (\vol \cE-\vol \cQ_1).
$$

We have computed:

\begin{align*}
	\vol\cQ_1 =& \frac{\mbox{\scalebox{0.8}{155143659305367638658204514673150261711154597948604269685210422288200009}}}{\mbox{\scalebox{0.8}{1102320838271070278766883635115881896290018550251848550368411648000000000}}},\\
	\vol\cQ_2 =& \frac{\mbox{\scalebox{0.8}{8007917191946827148905632396266883808060150761021309697108559220076039}}}{\mbox{\scalebox{0.8}{1653481257406605418150325452673822844435027825377772825552617472000000000}}},\\
	\vol\cQ_3 =& \frac{\mbox{\scalebox{0.8}{2072705500667484952215435851434572363770941977453049707343465792912717}}}{\mbox{\scalebox{0.8}{16534812574066054181503254526738228444350278253777728255526174720000000000}}}.
\end{align*}

Both ways of computing the probability that the Condorcet winner is the same as the plurality winner, but not the Borda winner, deliver the same result, that is:
$$\frac{\mbox{\scalebox{0.8}{82151877420135756441271759814103410444372449587666146678429057993673107}}}{\mbox{\scalebox{0.8}{1377901047838837848458604543894852370362523187814810687960514560000000000}}}\approx 0.059621.$$

\subsection{Condorcet efficiency of elimination} In this subsection we study the  \emph{Condorcet efficiency} of elimination procedures.
This is the conditional probability that the Condorcet winner, provided that such winner
exists, is elected by a certain voting scheme, as the number of voters $k\to\infty$. We consider the following two voting schemes.

First, \emph{the plurality elimination rule}: this is an iterative procedure, in which, at
each voting step, the candidate who obtained the minimum number of first place votes is
eliminated. The last candidate non eliminated is the winner. Second, \emph{the negative plurality elimination rule}: similarly, at each voting step the candidate with the maximum number of last place votes is eliminated.
For four candidates elections both lead to three-stage elimination procedures, thus our study here completes the data presented in Table 7.4 of \cite{GL2}.

It seems that the simplest way to compute this probability is to consider the complementary event, that is the event that the Condorcet winner is eliminated either in the first or the second round. Notice that, if the Condorcet winner will pass through the first and the second round, then he or she will automatically win the third round, so the study of the third round is not needed.

The outcome that the Condorcet winner is eliminated in the first round of a certain voting scheme is called \emph{the (reverse) strong Borda paradox} and its study was first introduced
by the Chevalier de Borda in \cite{B}. The occurrence of the (reverse) strong Borda paradox under both the plurality rule and the negative plurality rule was fully studied in subsection 2.5 of
\cite{BIS2} to where we refer the reader for details and we only recall here that the volume of the polytope associated with the strong Borda paradox is
$$
\vol \cB_{\Sg}=\frac{325451674835828550681491}{68475651442606080000000000} = \vol \cB_{\RevNPR},
$$
while the volume of the polytope associated with reverse strong Borda paradox is
$$
\vol \cB_{\SgRev}=\frac{104898234852130241}{21035720123168587776} = \vol \cB_{\RevPR}.
$$
We also refer the reader to \cite[Remark 3 (a)]{BIS2} for extra needed details, which will clarify the second notation used.

Lest us denote by $\cF$ the polytope corresponding to the event that candidate $A$ is the Condorcet winner, candidate $D$ is eliminated in the first round and candidate $A$ is eliminated in the second round.

Then, the Condorcet efficiency of both elimination rules may be computed with the formula
$$
\frac{p_{A=\CW}-\vol \cB_{\Rev}-3\vol \cF }{p_{A=\CW}},
$$
where $\vol \cB_{\Rev}$ should be replaced by $\vol \cB_{\RevPR}$ and $\vol \cB_{\RevNPR}$.

We have obtained
$$
\vol \cF_{\PR}=\frac{6537508029403236323215409545161316879405265171603}{1989889702166773519891328549909849702400000000000000},
$$
so that the Condorcet efficiency of the plurality elimination rule turns out
to be
$$
\frac{129178312275188795293522359266689257253407234828397}{139023462671726486558162887377734860800000000000000}
\approx 0.929184,
$$
while
$$
\vol \cF_{\NPR}=\frac{87391394898401644146716674012811354620163132417}{31091026140009682822081785811945799024640000000000},
$$
so that the Condorcet efficiency of the negative plurality elimination rule turns out
to be
$$
\frac{2035523745603707762358521726967860659560986470207}{2172171707454289770732195078088823930880000000000}
\approx 0.937092.
$$

It may come as a surprise the fact that the Condorcet efficiency of the negative plurality elimination rule is greater
than the Condorcet efficiency of the plurality elimination rule. This is not totally unexpected, considering the data presented in
Table 7.4 of \cite{GL2} for three candidates two-rounds elimination procedures. However, in order to check our results, we have computed the probabilities for all ten possible results that
the Condorcet winner may obtain in  the three-rounds elimination procedures. The approximative numbers are contained in the tables below.
For space reasons we did not include here the full exact data, which is available on request from the authors.

For the plurality elimination rule the probabilities are contained in Table \ref{PR_data}.
\begin{table}[hbt]
	\centering
	\begin{tabular}{|c|c|c|c|c|}
		\hline
		\multicolumn{2}{ |c| }{}   & \multicolumn{3}{ |c| }{2-nd round} \\
		\cline{3-5}
		\multicolumn{2}{ |c| }{}  & I & II  & III  \\ \hline
		&   I   & 0.69605467532 & 0.04320695864  & 0.00335247384       \\ \cline{2-5}
		\strut 1-st round            &   II  & 0.06678615010 & 0.08902016651  & 0.01327777245       \\ \cline{2-5}
		&   III & 0.01396067951 & 0.02015490336  & 0.03039424802       \\ \hline
		
	\end{tabular}
	\vspace*{2ex} \caption{Probabilities under PR}\label{PR_data}
\end{table}

For the negative plurality elimination rule the probabilities are contained in Table \ref{NPR_data}.
\begin{table}[hbt]
\centering
\begin{tabular}{|c|c|c|c|c|}
\hline
\multicolumn{2}{ |c| }{}   & \multicolumn{3}{ |c| }{2-nd round} \\
\cline{3-5}
\multicolumn{2}{ |c| }{}  & I & II  & III  \\ \hline
&   I   & 0.46569938269    & 0.07611279571    & 0.00978979031       \\ \cline{2-5}
\strut 1-st round            &   II  & 0.16256921634    & 0.11815379945    & 0.01272253146       \\ \cline{2-5}
&   III & 0.04072126773    & 0.07383508505    & 0.01771994652       \\ \hline

\end{tabular}
\vspace*{2ex} \caption{Probabilities under NPR}\label{NPR_data}
\end{table}

The entries in both tables should be read as follows: the entry at row $i$ and column $j$ represents an approximation of the
conditional probability that the Condorcet winner obtains the $i$-th place in the first round and $j$-th place in the second round, under the assumption that such a winner exists.
The missing number is the conditional probability that the Condorcet winner is eliminated in the first round, or in other words the probability of the (corresponding) reverse Borda paradox.
Those probabilities have been computed in subsection 2.5 of \cite{BIS2} (they are also reported in \cite[Table 7.5]{GL2}). More precisely,
the probability of the reverse strong Borda paradox under the plurality rule (respectively the negative plurality rule) is $\frac{104898234852130241}{4408976007260798976}\approx 0.02379$
(respectively $\frac{325451674835828550681491}{14352135440302080000000000}\approx 0.02268$). The exact numbers obtained add perfectly, in both cases studies the sum of all ten numbers equals $1$.

\begin{remark}\label{3rounds} The examples in this subsection are also computable by pyramid decomposition and symmetrization as discussed in \cite{BIS} and \cite{BS}.	 However, this will take several weeks on a quite powerful system in place of minutes on a rather standard computer.
\end{remark}

\end{document}